 \documentclass[11pt]{amsart}
\usepackage{geometry}   
\usepackage{setspace}
\usepackage{graphicx}
\usepackage{amssymb, amsthm}
\usepackage{epstopdf}
 \newtheorem{theorem}{Theorem}[section]
    \newtheorem{corollary}[theorem]{Corollary}
   \newtheorem{lemma}[theorem]{Lemma}
    \newtheorem{proposition}[theorem]{Proposition}

    \theoremstyle{definition}
\newtheorem{definition}[theorem]{Definition}
\newtheorem{remark}[theorem]{Remark}
\newtheorem*{claim}{Claim}
\newtheorem{example}[theorem]{Example}
\newtheorem{examples}[theorem]{Examples}

  \newcommand{\Z}{\ensuremath{{\mathbb{Z}}}}

  \newcommand{\G}{\Gamma}

\newcommand{\AG}{A_\G}            

\newcommand{\Lv}{lk[v]}

\title{ Subgroups and quotients of automorphism groups of RAAGS}
\author{Ruth Charney and Karen Vogtmann}

\thanks {R. Charney was partially supported by NSF grant DMS 0705396.  K. Vogtmann was partially supported by NSF grant DMS 0705960}

\begin{document}
\maketitle

\begin{abstract}  We study subgroups and quotients of outer automorphsim groups of right-angled Artin groups (RAAGs).  We prove that for all RAAGs, the outer automorphism group is residually finite and, for a large class of RAAGs, it satisfies the Tits alternative.  We also investigate which of these automorphism groups contain non-abelian solvable subgroups.  
\end{abstract}

\section{Introduction}
 
A right-angled Artin group, or RAAG,  is a finitely-generated group determined completely by the relations that some of the generators commute.  A RAAG is often described by giving a simplicial graph $\G$ with one vertex for each generator and one edge for each pair of commuting generators.   RAAGs include free groups  (none of the generators commute) and free abelian groups (all of the generators commute).  Subgroups of free groups and free abelian groups are easily classified and understood, but  subgroups of right-angled Artin groups lying between these two extremes  have proved to be a rich source of examples and counterexamples in geometric group theory.  For details of this history, we refer to  the article \cite{Ch07}.   

Automorphism groups and outer automorphism groups of RAAGs have received less attention than the groups themselves, with the notable exception of the two extreme examples, i.e. the groups $Out(F_n)$ of outer automorphism groups of a free group and the general linear group $GL(n,\Z)$.   The group  $Out(F_n)$ has been shown to share a large number of properties with $GL(n,\Z)$, including several kinds of finiteness properties and the Tits alternative for subgroups.  These groups have also been shown to differ in significant ways, including the classification of solvable subgroups.  In a series of recent papers \cite{CCV07,ChVo, BCV09}, we have begun to address the question of which properties shared by $Out(F_n)$ and $GL(n,\Z)$ are in fact shared by the entire class of outer automorphism groups of right-angled Artin groups. We are also interested in the question of determining properties which depend on the shape of $\G$ and in determining exactly how they depend on it.  

In our previous work, an  important role was played by certain restriction and projection homomorphisms, which allow one to reduce questions about the full outer automorphism group of a RAAG to questions about the outer automorphism groups of smaller subgroups.  
In the first section of this paper we recall these tools and develop them further.  In the next section we apply them to prove 

\medskip
\noindent {\bf Theorem~\ref{RF}.}  {\it For any defining graph $\G$, the group $Out(\AG)$ is residually finite.}
\medskip

This result was obtained independently by A. Minasyan \cite{Mi09}, by different methods.  We next prove the Tits' alternative for a certain class of {\it homogeneous} RAAGs (see section~\ref{homogeneous}). 

\medskip
\noindent {\bf Theorem~\ref{TA}.}  {\it If $\G$ is homogeneous, then $Out(\AG)$ satisfies the Tits' alternative.}
\medskip

In the last section, we investigate solvable subgroups of $Out(\AG)$.  We provide examples of non-abelian solvable subgroups and we determine an upper bound on the virtual derived length of solvable subgroups when $\AG$ is homogeneous. Finally, by studying translation lengths of infinite order elements, we find conditions under which all solvable subgroups of $Out(\AG)$ are abelian.  
We show that excluding ``adjacent transvections" from the generating set of  $Out(\AG)$ gives rise to a subgroup  $\widetilde Out(\AG)$ satisfying a strong version of the Tits alternative.

\medskip
\noindent {\bf Corollary~\ref{strong Tits}.}  {\it If  $\G$ is homogeneous of dimension $n$, then every subgroup of $\widetilde Out(\AG)$ is either virtually abelian or contains a non-abelan free group.}
\medskip

Thus for graphs which do not admit adjacent transvections, the whole group $Out(\AG)$ satisfies this property.  One case which is simple to state is the following.
 
\medskip
\noindent {\bf Corollary~\ref{2Dim}.}  {\it If $\G$ is connected with no triangles and no leaves, then all solvable subgroups of $Out(\AG)$ are virtually abelian.}
\medskip

Charney would like to thank the Forschungsinstitut f\"ur Mathematik in Zurich and Vogtmann the Hausdorff Institute for Mathematics in Bonn for their hospitality during the writing of this paper.  Both authors would like to thank Talia Fern\'os for helpful conversations.

\section{Some combinatorics of simplicial graphs}

 Certain combinatorial features of the defining graphs $\G$ for our right-angled Artin groups will be important for studying their automorphisms.  In this section we establish notation and recall some basic properties of these features.   

\begin{definition} Let $v$ be a vertex of $\G$. The {\it link of $v$}, denoted $lk(v),$ is the full subgraph spanned by all vertices adjacent to $v$.  The {\it star of $v$}, denoted $st(v),$ is the full subgraph spanned by $v$ and $lk(v)$.
\end{definition}

\begin{definition}Let $\Theta$ be a subgraph of $\G$.   The {\it link} of $\Theta$, denoted $lk(\Theta),$   is the intersection of the links of all vertices in $\Theta$.  The {\it star} of $\Theta$, denoted $st(\Theta)$   is the full subgraph spanned by $lk(\Theta)$ and $\Theta$. The {\it perp} of $\Theta$, denoted $\Theta^\perp$, is the intersection of the stars of all vertices in $\Theta$.  (See Figure 1.)
\end{definition}

\begin{figure}
\begin{center}
\includegraphics[width=4in]{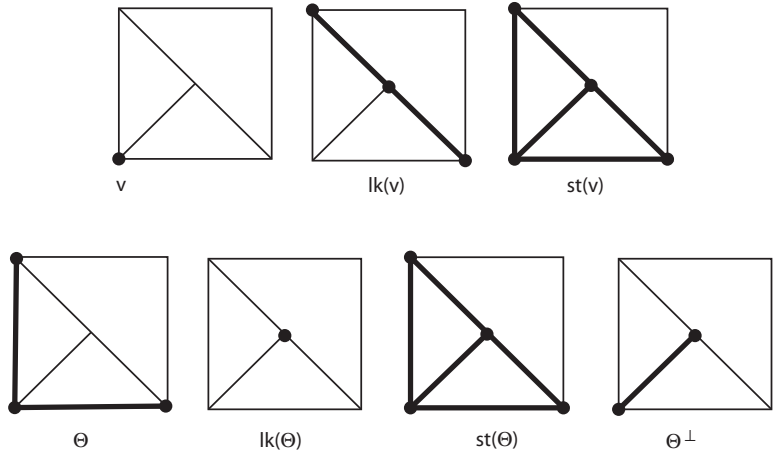}
\end{center}
\caption{Links, stars and perps} 
\end{figure}

These can be expressed in terms of distance in the graph as follows:
\begin{itemize}
\item $v\in lk(\Theta)$ iff $d(v,w)=1$ for all $w\in\Theta$
\item $v\in \Theta^\perp$ iff $d(v,w)\leq 1$ for all $w\in\Theta$
\item $v\in st(\Theta)$ iff $v\in lk(\Theta)\cup\Theta$
\end{itemize}
 
Recall that a complete subgraph of $\G$  is called  a {\it clique}.  (In this paper, cliques need not be maximal.) 
 If $\Delta$ is a clique, then $st(\Delta)=\Delta^\perp$; otherwise $st(\Delta)$ strictly contains $\Delta^\perp$.  

 \begin{lemma}\label{clique}   If $\Delta$ is a clique, then $st(\Delta)^\perp$ is also a clique and  
 $st(\Delta) \supseteq st(\Delta)^\perp \supseteq \Delta$.  
 \end{lemma}
 \begin{proof}  Since $\Delta$ is a clique, $v\in st(\Delta)$ implies $st(v)\supseteq\Delta$.  Therefore $$st(\Delta)^\perp=\cap_{v \in st(\Delta)} st(v)\supseteq \Delta.$$
 
 If $x\in st(\Delta)^\perp$, then $d(x,v)\leq 1$ for all vertices $v\in st(\Delta)$, including all $v\in\Delta$, i.e. $x\in st(\Delta)$.  If $y$ is another vertex in $st(\Delta)^\perp$, then similarly $d(y,v)\leq 1$ for all vertices $v\in st(\Delta)$,  so in particular $d(y,x)=1$.  Since any two vertices of $st(\Delta)^\perp$ are adjacent, $st(\Delta)^\perp$ is a clique.
 
 \end{proof}

We define  $v\leq w$ to mean $lk(v)\subseteq st(w)$.  This relation is transitive and induces a partial ordering on equivalence classes of vertices $[v]$, where $w \in [v]$  if and only if $v\leq w$ and $w\leq v$ (\cite{ChVo}, Lemma 2.2).  The links  $lk[v]$ and stars $st[v]$ of equivalence classes of maximal vertices $v$ will be of particular interest to us.  

\begin{remark}  In the authors' previous paper \cite{ChVo}, the notation $J_{[v]}$ was used to denote the star of an equivalence class $[v]$. This notation was chosen to emphasize that $st[v]$ has the  structure of the ``join" of two smaller graphs, $[v]$ and $lk[v]$.  In the current, more general setting, we find the notation $st(\Theta)$ to be more intuitive. 
\end{remark}

 For a full subgraph $\Theta \subset \G$, the right-angled Artin group $A_\Theta$ embeds into $A_\Gamma$ in the natural way.  The image is called a {\it special subgroup} of $A_\G$, and  we use the same notation $A_\Theta$ for it.  An important observation is that the centralizer of $A_\Theta$ is equal to $A_{\Theta^\perp}$ (see, e.g.,  \cite{CCV07}, Proposition 2.2).  
 
 We remark that if $v$ is a vertex in $\Theta \subset \G$, then it is possible for $v$ to be maximal in $\Theta$ but not  in $\G$.  Unless otherwise stated, the term "maximal vertex" will always mean maximal with respect to the original graph $\G$. 
 
 The subgraph spanned by $[v]$ is either a clique, or it is disconnected and discrete (\cite{ChVo}, Lemma 2.3).  In the first case  the subgroup $A_{[v]}$ is abelian and we call $v$ an {\it abelian vertex}; in the second,    $A_{[v]}$ is a non-abelian free group, and we call $v$ a {\it non-abelian vertex}.  
Note that for any vertex $v$, $st[v]$ is  the union of the stars of the vertices $w \in[v]$. 

A {\it leaf} of $\G$ is a  vertex which is an endpoint of  only one edge.   A {\it leaf-like} vertex is a vertex  $v$ whose link contains a unique maximal vertex $w$, and $[v]\leq [w]$.  In particular, a leaf is leaf-like.   If $\G$ has no triangles, then every leaf-like vertex is in fact a leaf.

\section{Key tools}

Generators for $Out(\AG)$ were determined by M. Laurence \cite{Lau95}, extending work of H. 
Servatius \cite{Ser89}.  They consist of 
\begin{itemize}
\item graph automorphisms 
\item inversions of a single generator  $v$ 
\item  transvections  $v\mapsto vw$ for generators $v\leq w$
\item partial conjugations by a generator $v$ on one component  of $\G-st(v)$ 
\end{itemize}
As in \cite{ChVo}, we consider the finite-index subgroup $Out^0(\AG)$ of $Out(\AG)$ generated by inversions, transvections and partial conjugations.  This is a normal subgroup, called the {\it pure outer automorphism group}

If $\G$ is connected and $v$ is a maximal vertex, then any pure outer automorphism $\phi$ of $\AG$ has a representative $f_v$ which preserves both $A_{[v]}$ and $A_{st[v]}$ (\cite{ChVo}, Prop. 3.2).  This allows us to define several maps from $Out^0(\AG)$ to the outer automorphism groups of various special subgroups, as follows. 

\begin{enumerate}
\item Restricting $f_v$ to $A_{st[v]}$  gives a {\it restriction map}
$$R_v\colon Out^0(\AG)\to Out^0(A_{st[v]}).$$
\item  The map $\AG\to A_{\G-[v]}$ which sends each generator in $[v]$ to the identity induces an {\it exclusion map} $$E_v\colon Out^0(\AG)\to Out^0(A_{\G-[v]}).$$  
\item   Since $v$ is maximal with respect to the graph $st[v]$ and $lk[v]=st[v]-[v]$, we can compose the restriction map on $A_\G$ with the exclusion map on $A_{st[v]}$ to get a {\it projection map}   $$P_v\colon Out^0(\AG)\to Out^0(A_{lk[v]}).$$ 
\end{enumerate}

If $\G$ is the star of a single vertex $v$, then  $[v]$ is the unique maximal equivalence class, and $R_v$ is the identity.  
If $\G$ is a complete graph, then $\G=[v]$ and $lk[v]$ is empty, in which case we define $P_v=E_v$ to be the trivial map. 

The reader can verify that these maps are well-defined homomorpisms. For the restriction map this follows from the fact that $A_{st[v]}$ is its own normalizer.  For the exclusion map it follows from the fact that the normal subgroup generated by a maximal equivalence class $[v]$ is charactersitic.  (See  \cite{ChVo} for details).

\subsection{The amalgamated restriction homomorphism $R$}
Let $\G$ be a connected graph.  We can put all of the restriction maps $R_v$ together to obtain an amalgamaged restriction map $$R=\prod R_v\colon Out^0(\AG)\to\prod  Out^0(A_{st[v]}),$$
where the product is over all maximal equivalence classes $[v]$.
It was proved in \cite{ChVo} that  the kernel $K_R$ of $R$ is a finitely-generated free abelian group, generated by partial conjugations.  If $\G$ has no triangles, we also found a set of generators for $K_R$ \cite{CCV07}.   We will need this information for general $\G$ in what follows,   so we will now present another (and simpler) proof that $K_R$ is free abelian which also identifies a  set of generators for $K_R$.  The proof will use the following fact due to Laurence.

\begin{theorem}[\cite{Lau95}, Thm 2.2] \label{partialconj} 
An automorphism of $\AG$ which takes every vertex to a conjugate of itself is a product of partial conjugations.
\end{theorem}

By definition, any automorphism representing an element of   $K_R$  acts on the star of each maximal equivalence class of vertices as conjugation by some element of $\AG$.  We begin by showing that the same is true for every equivalence class:

\begin{lemma}\label{star} Let $f$ be an automorphism representing an element of $K_R$.  Then for every vertex $v\in\G$, $f$ acts on  $st[v]$ as conjugation by some $g\in \AG$.  
\end{lemma}
 \begin{proof}   This is by definition of the kernel if $v$ is maximal.   Since every vertex of $\G$ is in the star of some maximal vertex, $f$ sends every vertex to a conjugate of itself.  
 By Theorem~\ref{partialconj}, this implies that $f$ is a product of partial conjugations.  
 
 If $v$ is not maximal, then choose a maximal vertex $v_0$ with $v<v_0$.  After adjusting by an inner automorphism if necessary, we may assume $f$ is the identity on $st[v_0]$.  If $v$ is adjacent to $v_0$, then $st[v]\subset st[v_0]$ and we are done.  
 
If $v$ is not adjacent to $v_0$, choose a maximal vertex $w_0\in lk(v)\cap lk(v_0)$ (note that one always exists).  Then  $f$ acts as conjugation by some $g$ on $st[w_0]$.  Let $e_0$ be the edge from $v_0$ to $w_0$.  Since $st(e_0)\subset st(w_0)$, $f$ acts as conjugation by $g$ on all of $st(e_0)$. Since  $st(e_0)\subset st(v_0)$,  $g$ centralizes $st(e_0)$, i.e. $g$ is in the subgroup generated by $st(e_0)^\perp$.   By Lemma~\ref{clique},  $st(e_0)^\perp=\Delta$ is a clique containing $e_0$, so the subgroup $A_\Delta$ is abelian.  

\begin{figure}\label{starfigure}
\begin{center}
\includegraphics[width=2.5in]{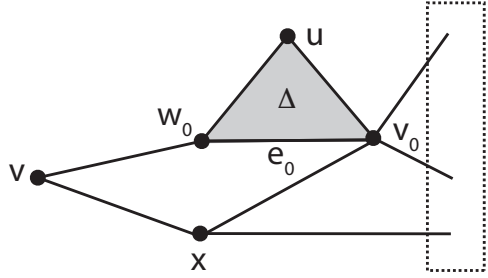}
\end{center}
\caption{ Notation for  proof of Lemma~\ref{star} }
\end{figure}

Since $A_\Delta$ is abelian, we can write $g = g_2g_1$ where $g_1$ is a product of generators in $lk[v]$ and $g_2$ a product of generators not in $lk[v]$. We claim that $f$ acts as conjugation by $g_2$ on all of $st[v]$.  Since  $[v] \subset st[w_0]$, $f$ acts as conjugation by $g$ on $[v]$, and since $g_1$ commutes with $[v]$, this is the same as conjugation by $g_2$.  The action of $f$ on $lk[v]$ is trivial, since $lk[v] \subset st[v_0]$, so it suffices to show that $g_2$ commutes with $lk[v]$.
For suppose $u \in \Delta$ does not lie in  $lk[v]$, and $x \in lk[v]$.  Then either $x$ lies in $st(u)$, hence commutes with $u$, or  $x$ and $v$ lie in the same component of $\Gamma -  st(u)$.  In the latter case, since $f$ is a product of partial conjugations, the total exponent of $u$ in the conjugating element must be the same at $v$ and at $x$; but $f(x)=x$, so this total exponent must be $0$.  That is, $u$ can appear as a factor in $g_2$ only if it commutes with all of $lk[v]$.
 \end{proof}

Next, we describe some automorphisms contained in the kernel $K_R$.  If $\G$ is a connected graph and $v$ is a vertex of $\G$, say vertices $x$ and $y$ are in the same {\it $\hat v$-component} of $\G$ if $x$ and $y$ can be connected by an edge-path which contains no edges of $st(v)$ (though it may contain vertices of $lk(v)$). A $\hat v$-component lying entirely inside $st(v)$  is called a {\it trivial} $\hat v$-component, and any other $\hat v$-component is {\it non-trivial}.  In Figure 2,  there are two non-trivial $\hat v$ components, one consisting of $A\cup B$, and one consisting of $C\cup D$.   
If $st(v)$ has no triangles, a non-trivial $\hat v$-component is the same thing as a non-leaf component of $\G-v$.  In general,
each component of $\G-st(v)$ is contained in a single $\hat v$-component, but a single $\hat v$-component may contain several components of $\G-st(v)$.

\begin{figure}\label{vhatcomponents}
\begin{center}
\includegraphics[width=2.5in]{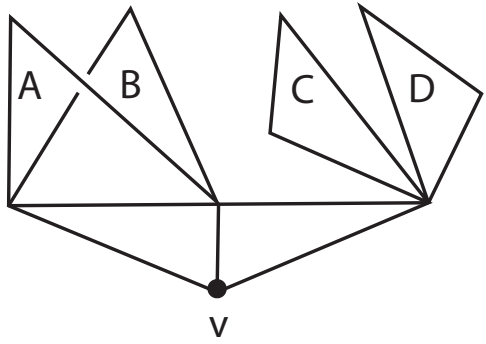}
\end{center}
\caption{ The $\hat v$-components are $A\cup B$ and $C\cup D$ }
\end{figure}

\begin{definition} A {\it $\hat v$-component conjugation} is an automorphism of $\AG$ which conjugates all vertices in a single nontrivial $\hat v$-component of $\G$ by $v$. 
\end{definition}
By the remarks above, a  $\hat v$-component conjugation is in general a product of partial conjugations by $v$ on components of $\G-st(v)$.  To see that such conjugations lie in $K_R$, note that for any $w$,   all of the vertices of $st[w]$ which do not lie in $st(v)$ lie in the same $\hat v$-component as $w$. Hence any  $\hat v$-component conjugation acts as an inner automorphsim on $st[w]$.  

Let $\hat c(v)$ be the number of non-trivial $\hat v$-components in $\G$. 
\begin{theorem}\label{kernel}  The kernel $K_R$ of the restriction map  is free abelian, generated by non-trivial $\hat v$-component conjugations for all $v\in\G$.  The rank of $K_R$ is $\sum_{v \in \G} (\hat c(v)-1).$
\end{theorem}

\begin{proof}  Let $\widehat{\mathcal PC} $ denote the set of all non-trivial $\hat v$-component conjugations for all $v\in\G$.  We first prove that $\widehat{\mathcal PC}$ generates $K_R$.

Let $\phi\in K_R$.  For each representative  $f$ of $\phi$, let $V_f$ be the set of vertices $v$ such that $st[v]$ is pointwise fixed.  Now fix a representative $f$ such that $V_f$ is of maximal size.  We proceed by induction on the number of vertices in $\G - V_f$.

If $V_f =\G$, then $f$ is the trivial automorphism and there is nothing to prove.  If not, choose a vertex $w$ at distance 1 from $V_f$, so $w$ is connected by an edge $e$ to some $v \in V_f$.  Then $f$ acts non-trivially on $st[w]$ as conjugation by some $g\in \AG$.  Since $f$ acts trivially on $st[v]$,   $g$ fixes $st(e)$.   The centralizer of $A_{st(e)}$ is equal to $A_{st(e)^\perp}$.  By Lemma~\ref{clique}, $st(e)^\perp=\Delta$ for some clique $\Delta$ containing $e$, so  $g$ is in the abelian subgroup $A_{\Delta}$ and we can write  $g=u_1^{\epsilon_1}\ldots u_k^{\epsilon_k}$ for  distinct vertices $u_i\in\Delta$.  

\begin{figure}\label{kernelfigure}
\begin{center}
\includegraphics[width=3in]{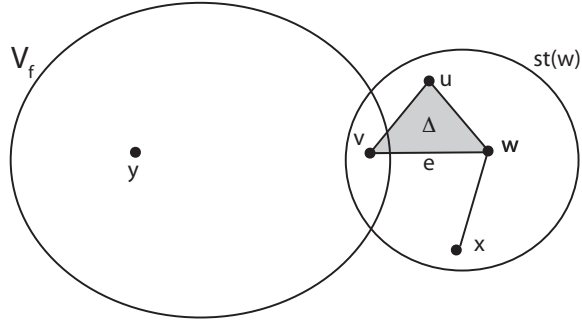}
\end{center}
\caption{ Notation for  proof of Theorem~\ref{kernel} }
\end{figure}
 
 If $st(w)\subseteq st(u_i)$ (e.g. if $u_i=w$), conjugation by $u_i$ is trivial on $st(w)$ and we may assume $\epsilon_i=0$, i.e. $u_i$ does not appear  in the expression for $g$.  If $V_f\subset st(u_i)$, replace $f$ by $f$ composed with  the inner automorphism by $u_i^{-\epsilon_i}$; the new $V_f$ contains (so is equal to) the old one.  We may now assume neither $st(w)$ nor $V_f$ are contained in the star of any $u_i$.  
  
Fix $u_i$ and $x\in st(w)-st(u_i)$ and $y\in V_f-st(u_i)$.   We claim that $x$ and $y$ are in different connected components of $\G-st(u_i)$.  To see this, suppose $x$ and $y$ are in the same connected component of $\Gamma-st(u_i)$.  Since $f$ sends each vertex to a conjugate of itself, Theorem~\ref{partialconj} implies that $f$ is a product of partial conjugations, hence $x$ and $y$ must be conjugated by the same total power of $u_i$.  For $y$ this power is zero, since $y\in V_f$, and so $\epsilon_i$ must also be zero, i.e. $u_i$ does not occur in the expression for $g$.  

We claim further that $x$ and $y$ must be in different $\hat u_i$-components of $\G$.  Suppose they were in the same $\hat u_i$-component.  Let $\gamma$  be an edge-path joining $y$ to $x$ which avoids edges of $st(u_i)$, with vertices  $y=x_0,x_1,x_2,\ldots, x_k=x$.   We know that $y$ is fixed by $f$ and $x$ is conjugated by a non-trivial power of $u_i$.  Therefore there is some $x_j$ in $lk(u_i)$ with the property that $x_{j-1}$ is not conjugated by $u_i$ but $x_{j+1}$ is conjugated by a non-trivial power of $u_i$.   Since $\gamma$ does not use edges of $st(u_i)$, neither $x_{j-1}$ nor $x_{j+1}$ is in $lk(u_i)$, i.e. neither commutes with $u_i$.  Thus $f$  does not act as conjugation by the same total power of $u_i$ on all of $st[x_j]$, contradicting Lemma~\ref{star}.   
 
The vertices of $st(w)-st(u_i)$ lie in a single, non-trivial $\hat u_i$-component (the component containing $w$) and by the discussion above, this $\hat u_i$-component contains no vertices of $V_f-st(u_i)$.
Thus, there is a non-trivial $\hat u_i$-component conjugation $f_i$  which  affects vertices of $st(w)$ but not  $V_f$.   The automorphism $f'=f_k^{-\epsilon_k}\circ  \cdots \circ f_1^{-\epsilon_1}\circ f$ has a strictly larger $V_{f'}$, which includes $w$ as well as $V_f$.  By induction, $f'$ is a product of elements of $\widehat{\mathcal PC}$, hence so is $f$.
 
 It remains to check that any two  elements of $\widehat{\mathcal PC}$ commute in $Out(\AG)$.  Let $f_v$ be  $\hat v$-component conjugation, and $f_w$   a  $\hat w$-component conjugation.  If $v$ and $w$ are adjacent, these commute.  If $d(v,w)>1$, then $st(w)$ is contained in a single  $\hat v$-component $D_v$, and  $st(v)$ is contained in a single  $\hat w$-component $D_w$.
 It follows that $D_v$ contains every $\hat w$-component except $D_w$, and $D_w$ contains every  $\hat v$-component except $D_v$.  It is now easy to check that for any $\hat v$-component $C_v$ and $\hat w$-component $C_w$, one of the following holds:  $C_w$ and $C_v$ are disjoint, $C_v\subset C_w$, $C_w\subset C_v$, or $\G-C_v$ and $\G-C_w$ are disjoint. In any of these cases, the corresponding partial conjugations  $f_v$ and $f_w$ commute in $Out(\AG)$.  
  
The only other relation among the generators  of  $\widehat{\mathcal PC}$ is that for a fixed $v$,  the product of all non-trivial $\hat v$-component conjugations is an inner automorphism.  The last statement of the theorem follows. 
\end{proof}

\subsection{The amalgamated projection homomorphism $P$}

We can combine the projection homomorphisms $P_v$ for maximal equivalence classes $[v]$ in the same way we combined the restriction homomorphisms, to obtain an amalgamated projection homomorphism
$$P=\prod P_v \colon Out^0(\AG)\to\prod  Out^0(A_{lk[v]}).$$

Recall that a vertex $v$  is called leaf-like if there is a unique maximal vertex $w$ in $lk(v)$ and this vertex satisfies $[v] \leq [w]$.  The transvection $v\mapsto vw$ is called a {\it leaf transvection}.  It is proved in \cite{ChVo} that the kernel $K_P$ of $P$ is a free abelian group generated by $K_R$ and the set of all leaf transvections.


 \section{Residual finiteness}

It is easy to see using congruence subgroups that $GL(n,\Z)$ is residually finite, and E. Grossman proved that $Out(F_n)$ is also residually finite (\cite{Gr}).   In  this section we use these facts together with our restriction and exclusion homomorphisms to show that in fact $Out(\AG)$ is residually finite for every defining graph $\G$.  The same result has been obtained by A. Minasyan \cite{Mi09} by different methods.  Both proofs use  a fundamental result of Minasyan and Osin which takes care of the case when the defining graph is disconnected:   

\begin{theorem} \cite{MiOs}\label{MiOs}  If $G$ is a finitely generated, residually finite group with infinitely many ends, then $Out(G)$ is residually finite.
\end{theorem}

\begin{theorem}\label{RF}  For any right-angled Artin group $\AG$, $Out(\AG)$ is residually finite.
\end{theorem}

\begin{proof}   Every right angled Artin group $\AG$ is finitely generated and residually finite (it's linear), and $\AG$ has infinitely many ends if and only if $\G$ is disconnected. Therefore,  by Theorem~\ref{MiOs}, we may assume that $\G$ is connected.

We proceed by induction on the number of vertices in $\G$.  

Consider first the case in which  $\G = st[v]$ for a single equivalence class $[v]$.  If $[v]$ is abelian, we know by Proposition 4.4 of \cite{ChVo}  that 
\[
Out(\AG)= Tr \rtimes (GL(A_{[v]}) \times Out(A_{\Lv}))
\]
where $Tr$ is the free abelian group generated by the leaf transvections.  Since $[v]$ is abelian, $GL(A_{[v]})=GL(k,\Z)$, which is residually finite, and $Out(A_{lk[v]})$ is residually finite by induction.  The result now follows because semi-direct products of finitely generated residually finite groups are residually finite \cite{Mi71}.   
If $[v]$ is non-abelian, then $Out(\AG)=Out(A_{[v]})\times Out(A_{lk[v]})$ (or possibly a $\Z/2\Z$-extenion of this).  Since $A_{[v]}$ is a free group, $Out(A_{[v]})$ is residually finite and $Out(A_{lk[v]})$ is residually finite by induction, so this case also follows.  

Now suppose that $\G$ is not the star of a single equivalence class.  Since $Out^0(\AG)$ has finite index in $Out(\AG)$, it suffices to prove that $Out^0(\AG)$ is residually finite.  For any maximal equivalence class $[v]$, $Out^0(st[v])$ is residually finite by induction, so any element of $Out^0(\AG)$ which maps non-trivially under $R$ is detectable by a finite group.  It remains to show that the same is true for elements in the kernel $K_R$  of $R$.  

Let $\phi$ be an element of $K_R$.  It follows from Lemma \ref{kernel} and the fact that $K_R$ is abelian that $\phi$ can be factored as
\[
\phi = \phi_1 \circ \cdots \circ \phi_k
\]
where $\phi_i$ is a product of $\hat v_i$-component conjugations, and the classes $[v_1], \dots ,[v_k]$ are distinct.
  Let $[w]$ be a maximal vertex adjacent to $[v_1]$.
Consider the image of $\phi$ under the exclusion homomorphism $E_w: Out^0(\AG) \to
Out^0(A_{\G - [w]})$.  By induction, the target group $Out^0(A_{\G-[w]})$ is residually finite, so it suffices to show that this image, $\bar\phi$,  is non-trivial.  

Write $\bar\phi=\bar\phi_1 \circ \cdots \circ \bar\phi_k$. Note that
$\bar\phi_1$ is still a nontrivial partial conjugation on $A_{\G - [w]}$ since the vertices which were removed commuted with all elements of $[v_1]$.  Moreover, for $i>1$,  the partial conjugations in $\bar\phi_i$ are either trivial, or are partial conjugations by elements distinct from $[v_1]$.  It follows that $\bar\phi$ acts non-trivially on $A_{\G - [w]}$ as required.
\end{proof}


\section{Homogeneous graphs and the Tits alternative}\label{homogeneous}

Recall that the Tits alternative for a group $G$ states that every subgroup of $G$ is either virtually solvable or contains a non-abelian free group.  Both $GL(n,\Z)$ and $Out(F_n)$ are known to satisfy the Tits alternative \cite{Ti72}, \cite{BFH00, BFH05}.  We will show that $Out(\AG)$ satisfies the Tits alternative for a large class of graphs $\G$.

\begin{definition}  Let $\G$ be a finite simplicial graph.  We say $\G$ is {\it homogeneous of dimension $0$} if it is empty, and {\it homogeneous of dimension $1$} if it is non-empty and discrete (no edges). For $n>1$, we say $\G$ is {\it homogeneous of dimension $n$} if it is connected and the link of every vertex is homogeneous of dimension $n-1$.   
\end{definition}

If $\Delta$ is a $k$-clique in $\G$ and $v$ is a vertex in $\Delta$, then the link of $\Delta$ in $G$ is equal to the link of $\Delta -v$ in $lk(v)$.  A simple inductive argument now shows that if $G$ is homogeneous of dimension $n$, then the link of any $k$-clique is  homogeneous of dimension $n-k$. In particular, every maximal clique in $\G$ is an $n$-clique (hence the terminology ``homogeneous").

 \begin{lemma}\label{galleries}
 If $\G$ is homogeneous of dimension $n>1$, then any two $n$-cliques $\alpha$ and $\beta$ are connected by a sequence of $n$-cliques $\alpha=\sigma_1, \sigma_2, \dots ,\sigma_k=\beta$
 such that $\sigma_{i-1} \cap \sigma_i$ is an $(n-1)$-clique.
 \end{lemma}
 
 \begin{proof}  
 We proceed by induction on $n$.  For $n=2$, this is simply the statement that $\G$ is connected.  For $n>2$, since $\G$ is connected and every edge is contained in an $n$-clique, we can find a sequence of $n$-cliques from $\alpha$ to $\beta$ such that consecutive $n$-cliques share at least a vertex.  Thus is suffices to consider the case where $\alpha$ and $\beta$ share a vertex $v$.
In this case, there are  $(n-1)$-cliques $\alpha'$ and $\beta'$ in $lk(v)$ that together with $v$ span
$\alpha$ and $\beta$.
By induction, $\alpha'$ and $\beta'$ can be joined by a sequence of $(n-1)$-cliques in $lk(v)$ that intersect consecutively  in $(n-2)$-cliques.  Taking the join of these with $v$ gives the desired sequence.
\end{proof}
 
We can also express this lemma in topological terms.  If $K_{\G}$ is the flag complex associated to $\G$ (that is, the simplicial complex whose $k$-simplices correspond to the $k$-cliques of $\G$), then the lemma states that for $\G$ homogeneous,  $K_{\G}$ is a chamber complex.

\begin{examples}  (1) For $n=2$, a graph $\G$ is homogeneous if and only if  $\G$ is connected and triangle-free. 
These are precisely the RAAGs studied in \cite{CCV07}.
(2) The join of two homogeneous graphs is again homogeneous so, for example, the join of two connected, triangle-free graphs is homogeneous of degree 4.   (3) If $\G$ is the 1-skeleton of a connected triangulated $n$-manifold, then $\G$ is homogeneous of dimension $n$.
 \end{examples}

Our main concern is to be able to do inductive arguments on links of vertices; in particular, we will need such links to be connected or discrete at all stages of the induction.  It may appear that homogeneity is a stronger condition than necessary.  This is not the case.

\begin{lemma}\label{connected}  $\G$ is homogeneous of dimension $n>1$ if and only if $\G$ is connected and the link of every (non-maximal) clique is either discrete or connected.
\end{lemma}

\begin{proof}  If $\G$ is homogeneous, then so is the link of every $k$-clique, $k<n$, so by definition it is either discrete or connected.  

Conversely, assume that $\G$ is connected and the link of every non-maximal $k$-clique is either discrete or connected.  We proceed by induction on  the maximal size $m$  of a clique in $G$.  If $m=2$, then the link of every vertex (1-clique) in $\G$ is discrete and non-empty, so by definition, $\G$ is homogeneous of dimension 2.  

For $m>2$, we claim first that the link of every vertex is connected.  For if $\G$ contains some vertex with a discrete link, then there exists an adjacent pair of vertices $v,w$ such that the link of $v$ is discrete while the link of $w$ is not.  In this case, $v$ lies in $lk(w)$ but $v$ is not adjacent to any other vertex in $lk(w)$.  This contradicts the assumption that the link of $w$ is connected.  

If $\Delta$ is a $k$-clique in $lk(v)$, then $\Delta \ast v$ is a $(k+1)$-clique in $\G$.  Since the link of $\Delta$ in $lk(v)$ is equal to the link of $\Delta\ast v$ in $\G$,  it is either discrete or connected.  Thus, by induction, $lk(v)$ is homogeneous.  Moreover, every link must be homogeneous of the same dimension, for if $v,w$ are adjacent vertices, then the homogeneous dimension of $lk(v)$ and $lk(w)$ are both equal to $r-1$ where $r$ is the size of the maximal clique containing $v$ and $w$.
\end{proof}

The next lemma contains some other elementary facts about homogeneous graphs.

\begin{lemma}\label{nostar} Let $\G$ be homogeneous of dimension $n$ and assume that $\G$ is not the star of a single vertex.   Let $[v]$ be a maximal equivalence class in $\G$.
\begin{enumerate}
\item If $[v]$ is abelian, then $[v]$ is a singleton.
\item For any maximal $[v]$, $lk[v]$ is homogeneous of dimension $n-1$ and is not the star of a single vertex.
\end{enumerate}
\end{lemma}

\begin{proof}  (1) 
Suppose $[v]$ is abelian and contains $k$ vertices. Then $[v]$ it spans a $k$-clique and $st[v]=st(v)$. 
By hypothesis, there is some $n$-clique $\sigma$ not contained in $st[v]$ and by Lemma~\ref{galleries}, we can choose $\sigma$ so that $\sigma \cap st[v]$ is an $(n-1)$-clique. It follows that  if $k>1$, 
then $\sigma$ contains some vertex of $[v]$ and hence every vertex of $[v]$ (since they are all equivalent),  contradicting our assumption that $\sigma$ does not lie in $st[v]$.  We conclude that $k=1$, or in other words, $[v]$ is a single point.  This proves (1).

For (2), let $[v]$ be any maximal equivalence class.  Then either $[v]$ is free, or a singleton and in either case, $lk[v]= lk(v)$, so it is homogeneous of dimension $n-1$.  If $lk[v]$ is contained in  the star of a single vertex $w\in lk[v]$, then $[v] < [w]$.  But this is impossible since $[v]$ is maximal. 
\end{proof}

We can now easily prove the main theorem of this section. 

\begin{theorem}\label{TA}  If $\G$ is homogeneous of dimension $n$, then $Out(\AG)$ satisfies the Tits alternative, that is, every subgroup of $Out(\AG)$ is either virtually solvable or contains a non-abelian free group.
\end{theorem}

\begin{proof}  For $\G$ a complete graph, $Out(\AG)=GL(n,\Z)$ so this follows from Tits' original theorem.  
So assume this is not the case.   It suffices to prove the Tits Alternative for the finite index subgroup $Out^0(\AG)$. We proceed by induction on $n$.  For $n=1$, $\AG$ is a free group and the theorem follows from \cite{BFH00, BFH05}.  

If $n>1$, then for every maximal $[v]$, $\Lv$ is homogeneous of lower dimension, so by induction, $Out^0(A_{\Lv})$ satisfies the Tits alternative.  It is easy to verify that the Tits alternative is preserved under direct products, subgroups, and abelian extensions, so the theorem now follows from the exact sequence 
$$1 \to K_P  \to Out^0(\AG) \to \prod Out^0(A_{\Lv}). $$
 \end{proof}


\section{Solvable subgroups}

\subsection{Virtual derived length}

\begin{definition} Let $G$ be a solvable group and $G^{(i)}$ its derived series.  The {\it derived length}  of $G$ is the least $n$ such that $G^{(n)}=\{1\}$.  The {\it virtual derived length} of $G$, which we denote by $vdl(G)$, is the minimum of the derived lengths of finite index subgroups of $G$.
\end{definition}

For an arbitrary group $H$, define
$$\mu (H) = \max \{vdl(G) \mid \textrm{$G$ is a solvable subgroup of $H$}\}. $$
Note that if $H$ is itself solvable, then $\mu(H)=vdl(H)$.

The following properties of $\mu (H)$ are easy exercises.
\begin{lemma}\label{mu-basics}  
\begin{enumerate} 
\item If  $H= \prod H_i$, then $\mu(H)= \max \{\mu(H_i)\}$.
\item If $N$ is a subgroup of $H$, then $\mu(N) \leq \mu(H)$. If $[H:N] < \infty$, then $\mu(N) = \mu(H)$.
\item If $N \lhd H$ is a solvable normal subgroup of derived length $k$, then $\mu(H) \leq \mu(H/N) + k$.
\end{enumerate}
\end{lemma}

A group $G$ has  $vdl(G)=1$ if and only if $G$ is virtually abelian, and hence  $\mu (H)=1$ if and only if every solvable subgroup of $H$ is virtually abelian.  By \cite{BFH05}, $\mu(Out(F_n))=1$ for any free group $F_n$.  

The situation for $GL(n,\Z)$ is more complicated.  Let $U_n$ denote the unitriangular matrices in  
$GL(n,\Z)$, that is, the (lower) triangular matrices with $1$'s on the diagonal.

\begin{proposition}   $\mu (U_n) =\lfloor \log_2(n-1)\rfloor +1$,  and
 $ \mu (U_n) \leq \mu(GL(n,\Z)) \leq \mu(U_n)+1.$
\end{proposition}

\begin{proof}  It is easy to
verify that $U_n(R)$ is solvable with derived length less than $log_2(n) +1$ for any ring $R$.  Let $e_{i,j}^a$ denote the elementary matrix with $a$ in the $(i,j)$-th entry.  For any finite index subgroup $G$ of $U_n$, there exists $m \in \Z$ such that $G$ contains all of the elementary matrices $e_{i,j}^m$ with $ i>j$.  The relation 
$$[e_{i,k}^m,e_{k,j}^m]=e_{i,j}^{m^2}$$
then implies that the $k$th commutator subgroup $G^{(k)}$ contains all of  the elementary matrices of the form $e_{i,j}^a$ with $i \geq j + 2^k$ and $a = m^{(2^k)}$. In particular, $G^{(k)}$ is non-trivial if 
$2^k < n$.  Thus the derived length of $G$ satisfies $\log_2(n)\leq dl(G)\leq dl(U_n) < log_2(n)+1$, which translates to the first statement of the proposition.  

The  first inequality of the second statement follows from Lemma \ref{mu-basics}(2). 
 For the second inequality, we use a theorem of Mal'cev \cite{Ma56}, which implies that every solvable subgroup $H \subset GL(n,\Z)$ is virtually isomorphic to a subgroup of $T_n(\mathcal O)$, the lower triangular matrices over the ring of integers $\mathcal O$ in some number field. 
 The first commutator subgroup of
 $T_n(\mathcal O)$ lies in $U_n(\mathcal O)$, so $vdl(H)\leq dl(T_n(\mathcal O))\leq dl(U_n(\mathcal O)) + 1=\mu(U_n)+1$.    
\end{proof}

\begin{remark}  The exact relation between $\mu(U_n)$ and $\mu(GL(n,\Z)$ is not completely clear.  Dan Segal has shown us examples demonstrating that  $\mu(GL(n,\Z))=\mu(U_n)+1$ for $n = 1 + 3\cdot 2^t$, while for $n= 1+ 2^t$ he shows   $\mu(U_n)=\mu(GL(n,\Z))$ \cite{Segal}.
\end{remark}

\subsection{Maximum derived length for   homogeneous graphs}
In the case of a homogeneous graph $\G$, it is easy to obtain an upper bound on the virtual derived length  of solvable subgroups of $Out(\AG)$: 

\begin{theorem}\label{mu}  Suppose $\G$ is homogeneous of dimension $n$.   Then $\mu (Out(\AG)) \leq n$.  
 \end{theorem}

\begin{proof}  If $\G$ is a complete graph, $Out(\AG)\cong GL(n,\Z)$, which has virtual derived length $\mu(GL(n,\Z)) < log_2(n) +2 \leq n+1$. So we may assume $\G$ is not a complete graph, in which case it contains a maximal vertex $v$ with  $lk[v]$ non-empty.  

 Since $Out^0(\AG)$  has finite index in $Out(\AG)$, their maximal derived lengths $\mu$ agree.    We proceed by induction on $n$.  For $n=1$, $\G$ is discrete so $\mu (Out(\AG))=1$ by \cite{BFH05}.  For $n>1$,  we apply Lemma \ref{mu-basics} to the abelian extension,
$$ 1 \to K_P \to Out^0(\AG) \to \prod Out^0(A_{\Lv}) $$
to conclude that $\mu(Out^0(\AG)) \leq 1+\max \{ \mu(Out^0(A_{\Lv})\}$.  By Lemma  \ref{nostar}, $\Lv$ is homogeneous of dimension $n-1$, so the theorem follows by induction.
\end{proof}

Here is a stronger formulation of the previous theorem.
If $\G$ is homogeneous of dimension $n$ and $\Delta$ is an $(n-1)$-clique, then $lk(\Delta)$  is discrete, hence generates a free group $F(lk(\Delta))$.

\begin{theorem}  Let $\G$ be homogeneous of dimension $n \geq 2$.  Then there is a homomorphism
$$Q:  Out^0(\AG) \to \prod Out (F(lk(\Delta))), $$
where the product is taken over some collection of $(n-1)$-cliques, such that the kernel of $Q$ is a solvable group of derived length at most $n-1$.
\end{theorem}

\begin{proof}  Induction on $n$.  For $n=2$, take $Q=P$.  The kernel $K_P$ is abelian.  

Suppose $n>2$.  Then $P$ maps 
$Out^0(\AG)$ to a product of groups $Out^0(A_{\Lv})$, where $\Lv$ is homogeneous of dimension $n-1$.  By induction,  there exists a homomorphism $Q_{v}$ from $Out^0(A_{\Lv})$ to a product of groups $Out (F(lk_{v}(\Delta)))$ where $\Delta$ is an $n-2$ clique in $lk(v)$ and $lk_{v}(\Delta)$ is its link.   The kernel $H_v$ of $Q_v$ is solvable of derived length at most $n-2$.  

Let $\Delta'=\Delta \ast v$.  Then $ \Delta'$ is an $(n-1)$-clique in $\G$ whose link $lk(\Delta')$ is exactly $lk_{v}(\Delta)$.  Thus the composite $Q= (\prod Q_v) \circ P$ gives the desired homomorphism. The kernel of $Q$ fits in an exact sequence
$$ 1 \to K_P \to ker~Q \to \prod H_v. $$
It follows that $ker~Q$ is solvable of derived length at most $n-1$.
\end{proof}

\subsection{Examples of solvable subgroups}

We now investigate lower bounds on the virtual derived length of $Out(\AG)$.  If $[v]$ is an abelian equivalence class with  $k$ elements,   then $GL(k,\Z)$ embeds as a subgroup of $Out(\AG)$; in particular, $Out(\AG)$ contains solvable subgroups of virtual derived length at least $log_2(k)$.    In homogeneous graphs, abelian equivalence classes have only one element, so one cannot construct non-abelian solvable subgroups in this way. However, non-abelian solvable subgroups do exist, and we show two ways of constructing them in this section.  More examples may be found in \cite{Da09}.

\begin{proposition}\label{lower1}    Let $\G$ be any finite simplicial graph.  Suppose $\G$ contains $k$ distinct vertices, $v_1, \dots ,v_k$ satisfying
\begin{enumerate}
\item $v_2, \dots ,v_k$ span a $(k-1)$-clique
\item $[v_1] \leq [v_2] \leq \dots \leq [v_k]$
\end{enumerate}
Then $Out(\AG)$ contains a subgroup isomorphic to the unitriangular group $U_{k}$.  In particular, 
$\mu (Out(\AG)) \geq log_2(k)$.
\end{proposition}

\begin{proof}  Let $\alpha_i$ denote the transvection $v_{i} \mapsto v_{i}v_{i+1}$.  Let $H$ denote the subgroup of $Out(\AG)$ generated by $\alpha_i$, $1 \leq i \leq k-1$.  Since $H$ preserves the subgraph $\G'$ spanned by the $v_i$'s, it restricts to a subgroup of $Out(A_{\G'})$.  It is easy to see that this restriction maps $H$  isomorphically onto its image, so without loss of generality, we may assume that $\G=\G'$.  

Abelianizing $\AG$ gives a map $\rho$ from $H$ to $GL(k,\Z)$.  The image of $\alpha_i$ under $\rho$  is the elementary matrix $e^1_{i+1,i}$. It follows that the image of $H$ in $GL(k,\Z)$ is precisely $U_k$.  Thus, it suffices to verify that the kernel of $\rho$ is trivial.  Let $\Delta$ be the clique spanned by $v_2, \dots ,v_k$.  Note that an element  of $H$ takes each $v_i$ to $v_iw_i$ for some $w_i \in A_\Delta$.  Since $A_\Delta$ is already abelian, a non-trivial $w_i$ cannot be killed by abelianizing $\AG$. Thus, $\rho$ is injective.  
\end{proof}

\begin{proposition}\label{lower2}    Let $\G$ be any finite simplicial graph.  Suppose $\G$ contains $k-1$ distinct  vertices, $v_1, \dots ,v_{k-1}$ satisfying
\begin{enumerate}
\item $v_1, \dots ,v_{k-1}$ span a $(k-1)$-clique
\item $[v_1] \leq \dots \leq [v_{k-1}]$
\item $\G - st(v_1)$ has at least two distinct components that are not contained in $st(v_{k-1})$
\end{enumerate}
Then $Out(\AG)$ contains a subgroup isomorphic to $U_{k}$. In particular, 
$\mu (Out(\AG)) \geq log_2(k)$.
\end{proposition}

\begin{proof}  For $i= 1, \dots k-2$, take $\alpha_i$,  to be the transvection $v_{i} \mapsto v_{i}v_{i+1}$. Let $C$ be a component of $\G - st(v_1)$ which is not contained in $st(v_{k-1})$.  For $i=1, \dots, k-1$, take $\beta_i$ to be the partial conjugation of $C$ by $v_i$.  Note that $\beta_i$ is non-trivial in $Out(\AG)$ since condition (3) guarantees that $\G-st(v_i)$ contains at least two components.  Let $H$ denote the subgroup of $Out(\AG)$ generated by the $\alpha_i$'s and $\beta_i$'s.  We claim that $H$ is isomorphic to $U_{k}$.

Since $\alpha_i$ acts only on vertices in $st(v_1)$ while $\beta_i$ acts only on vertices not in $st(v_1)$, the subgroups $H_{\alpha}$ and $H_{\beta}$ generated by the $\alpha_i$'s and $\beta_i$'s respectively, are disjoint and  $H_{\beta}$ is easily seen to be normal in $H$.  Hence $H$ is the semi-direct product , $H = H_{\beta} \ltimes H_{\alpha}$.  The subgroup $H_{\alpha}$ is isomorphic to $U_{k-1}$, as shown in the proof of the previous proposition, while $H_{\beta}$ is isomorphic to the free abelian group $A_\Delta$ generated by the clique $\Delta$ spanned by the $v_i$'s.   It is now easy to verify that $H$ is isomorphic to $U_{k}$ as claimed.
\end{proof}

\begin{example}  Suppose $\AG$ is homogeneous of dimension 2.  Then by Theorem \ref{mu},
$\mu(Out(\AG)) \leq 2$. In the next section (Corollary \ref{2Dim}), we will show that if $\G$ has no leaves, then $\mu(Out(\AG)) = 1$.  If $\G$ does have leaves, and for some leaf $v$, $st(v)$ separates $\G$, then
$\mu(Out(\AG)) = 2$. For if the components of $\G-st(v)$ are not leaves, then Proposition~\ref{lower2} implies that  $\log_2(3) \leq \mu(Out(\AG))$, and if some component is a leaf $v'$ attached at the same base $w$,  then Proposition~\ref{lower1} applied to $[v] \leq [v'] \leq [w]$ gives the same result. 
 \end{example}

\subsection{Translation lengths and solvable subgroups}

In constructing the non-abelian solvable subgroups above, a key role was played by transvections $v \to vw$ between adjacent (i.e. commuting) vertices.  We call these {\it adjacent transvections}.  In this section, we will show that without adjacent transvections, no such subgroups  can exist for homogenous graphs.

Recall that $Out(\AG)$ is generated by the finite set $S$ consisting of graph symmetries, inversions, partial conjugations and transvections. Define the following subsets of $S$,
\begin{align*}
\tilde S &=  S - \{ \mathrm{adjacent~transvections} \} \\
\tilde S^0 &= \tilde S - \{ \mathrm{graph~symmetries} \} 
\end{align*}
and the subgroups of $Out(\AG)$ generated by them,
\begin{align*}
\widetilde Out(\AG) &=\langle  \tilde S \rangle \\
\widetilde Out^0(\AG) &=\langle \tilde S^0 \rangle. 
\end{align*}
We will prove that for $\G$ homogeneous, all solvable subgroups of $\widetilde Out(\AG)$ are virtually abelian. 

The proof proceeds by studying the translation lengths of infinite-order elements.  The connection between solvable subgroups and translation lengths was first pointed out by Gromov \cite{Gromov}. 
\begin{definition} Let $G$ be a group with finite generating set $S$, and let $\|g\|$ denote the word length of $g$ in $S$.  The {\it translation length} $\tau(g)=\tau_{G,S}(g)$ is the limit
$$\lim_{k\to\infty}\frac{ \| g^k\|}{k}.$$
\end{definition}

Elementary properties of translation lengths include the following (see \cite{GS}, Lemma 6.2):
\begin{itemize}
\item  $\tau(g^k)=k\tau(g)$.
\item  If $S'$ is a different finite generating set, then $\tau_{G,S}(g)$ is positive if and only if $\tau_{G,S'}(g)$ is positive.  
\item If $H\leq G$ is a finitely-generated subgroup, and the generating set for $G$ includes the generating set for $H$, then $\tau_H(h)\geq \tau_G(h)$.  
\end{itemize}

Note that the kernel $K_R$ of the amalgamated restriction homomorphism lies in $\widetilde Out^0(\AG)$ since it it generated by products of partial conjugations.
\begin{proposition}\label{translation} Assume $\G$ is connected.  Then
 every element of the kernel $K_R$ of the amalgamated restriction homomorphism has positive translation length in $\widetilde Out^0(\AG)$.
\end{proposition}  

\begin{proof} Fix $\phi\in K_R$.  We will find a function $\lambda=\lambda_\phi\colon \widetilde Out^0(A_\G)\to \mathbb{R}_{\geq 0}$  satisfying
 \begin{enumerate}
 \item $\lambda(\phi^k)\geq \frac{k}{2} $ and 
 \item $\lambda(\gamma_1\ldots\gamma_k)\leq 2k$  for $\gamma_i\in \tilde S^0$
 \end{enumerate}  
The proof is then finished by the following argument.  Let  $m_k=\|\phi^k\|$, and write $\phi^k=\gamma_1\ldots \gamma_{m_k}$ with $\gamma_i \in  \tilde S^0$.   Then   
$$\frac{k}{2}\leq \lambda(\phi^k)= \lambda(\gamma_1\ldots \gamma_{m_k})\leq 2m_k$$ so
$$\frac{m_k}{k}\geq \frac{1}{4}>0$$ for all $k$, so $$\tau(\phi)=\lim_{k\to\infty} \frac{m_k}{k}\geq \frac{1}{4}>0.$$

To define $\lambda$ recall that by Theorem~\ref{kernel},  the kernel $K_R$ is free abelian,  generated by  $\hat w$-component conjugations.   We write $$\phi=\phi_{w_1}\phi_{w_2}\ldots \phi_{w_k}$$  where $\phi_{w_i}$ is a nontrivial product of conjugations by $w_i$ and the  $w_i$ are distinct. 

First observe that the only transvections onto $w_i$ are adjacent transvections.  For if $u$ is not adjacent to $w_i$ and $w_i \leq u$, then there is only one non-trivial $\hat w_i$-component (the component of $u$), hence the unique $\hat w_i$-component conjugation is an inner automorphism.  It follows that every element of $\widetilde Out^0(\AG)$ fixes $w_i$ up to conjugacy and inversion.

Set $w=w_1$.  For an arbitrary element $x\in\AG$, define $p(x)=p_w(x)$ to be the absolute value of the largest power of $w$ which can occur in a minimal-length word representing $x$.  For example, if $u$ or $v$ does not commute with $w$, then $p(wuvw^{-2})=2$.   If a minimal word representing $x$ does not contain any powers of $w$, then $p(x)=0$.

In \cite{HM95}, Hermiller and Meier describe a "left greedy" normal form for words in $\AG$, obtained by shuffling letters as far left as possible using the commuting relations and canceling inverse pairs whenever they occur.  In particular, any reduced word can be put in normal form just by shuffling.     It follows that the highest power of $w$ that can occur in a minimal word for $x$ is equal to the highest power of $w$ appearing in the normal form for $x$. 

For any automorphism $f\in Aut(\AG)$, define $p(f)$ to be the maximum over all vertices $v \neq w$ of  $p(f(v))$. For an  outer  automorphism $\phi\in \widetilde Out^0(\AG)$, define $\lambda(\phi)$ to be the minimum value of $p(f)$ as $f$ ranges over automorphisms $f$ representing $\phi$.  We must show that $\lambda$ satisfies properties (1) and (2) above.  

(1)  Let $f_w$ be a $\hat w$-component conjugation  on the $\hat w$ component $C$, and let $v\neq w$ be a vertex of $\G$.  If $v\in C-st(w)$, then $p(f^k_w(v))=k$, and   $p(f^k_w(v))=0$ otherwise.  An inner automorphism can reduce the power of $w$ by shifting it to vertices in the complement of $C$,  but cannot reduce the maximum power of $w$ over all vertices by more than $[k/2]$.  Since $\phi_w=\phi_{w_1}$ is non-trivial on at least one $\hat w$-component, this implies $\lambda(\phi_w^k)\geq k/2$. 
Since the partial conjugations $\phi_i$ for $i>1$ do not change the power of $w$ occuring at any vertex,  we conclude that $\lambda(\phi^k)\geq k/2$.

(2) To prove property (2) we need to first establish some properties of the power function $p$.
By abuse of notation, we will view $\tilde S^0$ as a subset of $Aut(\AG)$ in the obvious way.
\begin{claim}  Let $x\in \AG$.  If  $p(x)=0$ and $f\in \tilde S^0$
then $p(f(x))\leq 1$.
\end{claim} 

\begin{proof}
If $f$ is a transvection or partial conjugation by some $u \neq w$, then $p(f(x))=0$. Likewise for inversions.  So the only cases we have to consider are when $f$  is either  a non-adjacent transvection of $w$ onto $v$ or partial conjugation of $C\subset \G$ by $w$.  
 
Suppose $f$ is a (non-adjacent) transvection $f\colon v\mapsto vw$ or $f\colon v\mapsto wv$.  Then $f(x)$ has the property that any two copies of $w$ are separated by $v$ and any two copies of $w^{-1}$ are separated by $v^{-1}$.  ``Shuffling left" can never switch the order of $v$ and $w$, so this must also be true in the normal form for $f(x)$.   

 If $f$ is a partial conjugation by $w$, then the $w$'s in $f(x)$ alternate, i.e.  $$f(x)=a_1wa_2w^{-1}a_3 w \dots $$ where the $a_i$ are words which do not use $w$ or $w^{-1}$, so shuffling left can only cancel $w$-pairs, never increase the power to more than 1.
\end{proof}

A minimal word representing $x\in\AG$ can be put in the form $a_0w^{k_1}a_1w^{k_2}\ldots w^{k_n}a_n$ where
\begin{itemize}
\item
$a_i$ contains no $w$ and 
\item $w$ does not commute with $a_i$ for  $1 \leq i \leq n-1$.
\end{itemize}
so that $p(x)=max \{k_i\}$. 

\begin{claim}\label{px} For any $f\in \tilde S^0$ and $x\in \AG$, $p(f(x))\leq p(x)+2.$.
\end{claim}

\begin{proof}  First assume that $f(w)=w$.  This holds for all generators in $\tilde S^0$ with the exception of a partial conjugation by $u$ of a component $C$ containing $w$.  Write $x=a_0w^{k_1}a_1w^{k_2}\ldots,w^{k_n}a_n$ as above.  Let $b_i$ be the normal form for $f(a_i)$.  Then 
$f(x)=b_0w^{k_1}b_1w^{k_2}\ldots w^{k_n}b_n$, where $b_i$ does not commute with $w$.

Case 1:  $f$ is a partial conjugation or transvection by $w$.  Then no $w$ can shuffle across an entire $b_i$, so we need only consider the highest power appearing in $b_{i-1}w^{k_i}b_i$.  Now by the previous claim, $b_i$ is of the form 
$$b_i=c_1w^{\pm1}c_2w^{\pm 1}\ldots w^{\pm 1}c_k$$
where $w$ does not commute with $c_2,\ldots,c_{k-1}$ and similarly for $b_{i-1}$.  It follows that left shuffling of $b_{i-1}w^{k_i}b_i$ can at worst combine $w^{k_i}$ with the last $w$ in $b_{i-1} $ and the first $w$ in $b_i$, producing a power of at most $|k_i|+2$.

Case 2:  $f$ is a partial conjugation (on a component not containing $w$) or transvection by some $u\neq w$.  Then no new $w$'s appear and no $w$ can shuffle across an entire $b_i$,  so the maximum power of $w$ does not change, i.e., $p(f(x))=p(x)$.

It remains to consider the case where $f$ is a partial conjugation by $u$ with  $f(w)=uwu^{-1}$.  Then $f$ can be written as the composite of an inner automorphism by $u$ followed by a product of partial conjugations fixing $w$.  By case 2 above, we have $p(f(x))=p(uxu^{-1})$.  Since $u$ does not commute with $w$, conjugating by $u$ changes only the factors $a_0$ and $a_n$ in the normal form for $x$. Thus  $p(uxu^{-1})=p(x)$. 
\end{proof}

If $f\in Aut(\AG)$ can be written as a product of $m$ elements of $\tilde S^0$, then  the above claim shows that $p(f(v))\leq 2m$, for any vertex $v\neq w$.  If $\phi\in \widetilde Out^0(\AG)$ can be written as a product $\phi=\phi_1\ldots \phi_m$, with $\phi_i$ represented by $f_i \in \tilde S^0$, then $\lambda(\phi)\leq p(f)\leq 2m$.  This completes the proof of the proposition.
 \end{proof}

\begin{corollary}\label{positive} If  $\G$ is homogeneous of dimension $n$, then  the translation length of every infinite-order element of $\widetilde Out^0(\AG)$ is positive.
\end{corollary}

\begin{proof} We proceed by induction on $n$.  For $n=1$, $\AG$ is a free group $F_k$ and $\widetilde Out^0(\AG)=Out(F_k)$.  Alibegovic proved that infinite order elements of $Out(F_k)$ have positive translation length \cite{Al}.

For $n>1$, we will make use of the projection homomorphism  
$$P=\prod P_v \colon Out^0(\AG)\to \prod_{v \,\,maximal} Out(A_{lk[v]}).$$
Note that each projection $P_v$ maps generators in $\tilde S^0$ to either the trivial map or to a generator of the same form in $Out(A_{lk[v]})$. Thus the image of $\widetilde Out^0(\AG)$
lies in the product of the subgroups $\widetilde Out^0(A_{lk[v]})$.  Moreover, the kernel of $P$ restricted to  $\widetilde Out^0(\AG)$ is just $K_R$.  This follows from the fact that  $K_P$ is generated by $K_R$ and leaf-transvections, which by definition, are adjacent transvections.

By \cite{ChVo}, all of the groups we are considering are virtually torsion-free.
Let $G\leq \widetilde Out^0(\AG)$ be the inverse image of a torsion-free finite-index subgroup of $\prod \widetilde Out^0(A_{lk[v]})$.  If $\phi\in \widetilde Out^0(\AG)$ has infinite order, then some power of $\phi$ is a non-trivial (infinite-order) element in G,  so we need only prove that elements of  $G$ have positive translation length in $\widetilde Out^0(\AG)$. 

If the image of $\phi \in G$ is non-trivial in some $\widetilde Out^0(A_{lk[v]})$, then it has positive translation length by induction. If the image is trivial, then $\phi$ lies in $K_R$ so we are done by  Proposition~\ref{translation}.
\end{proof}

\begin{corollary}\label{strong Tits} If  $\G$ is homogeneous of dimension $n$, then $\widetilde Out(\AG)$  satisfies the strong Tits alternative, that is, every subgroup of $\widetilde Out(\AG)$ is either virtually abelian or contains a non-abelan free group.
\end{corollary}

\begin{proof}  By Theorem \ref{TA},  every subgroup not containing a free group is virtually solvable.  So it remains to show that every solvable subgroup is virtually abelian. 
Since $\widetilde Out^0(\AG)$ has finite index in $\tilde Out(\AG)$, it suffices to prove the same statement for $\widetilde Out^0(\AG)$.  

Bestvina \cite{Be}, citing arguments from Conner \cite{Co} and Gersten and Short \cite{GS},  shows that if a finitely-generated group is positive, virtually torsion-free and its abelian subgroups are finitely generated, then solvable subgroups must be virtually abelian.   $\widetilde Out^0(\AG)$ is positive by Corollary~\ref{positive} and virtually torsion-free by \cite{ChVo}.    Since $\G$ is homogeneous, the fact that abelian subgroups of $Out^0(\AG)$ are finitely generated follows by a simple induction from the same fact for $Out(F_n)$  and $GL(n,\Z)$ using the projection homomorphisms. 
 \end{proof}

In dimension 2, the only adjacent transvections are leaf transvections, so if $\G$ has no leaves, then 
$\widetilde Out(\AG)=Out(\AG)$.  Thus the following is a special case of Corallary~\ref{strong Tits}.

\begin{corollary}\label{2Dim} If  $\G$ is connected with no triangles and no leaves, then $Out(\AG)$ satisfies the strong Tits alternative.
\end{corollary}


\section{Questions}

Since the projection homomorphism $P\colon Out^0(\AG)\to \prod Out^0(A_{lk[v]})$ is defined only for connected graphs $\G$, inductive arguments using $P$   break down if the links of maximal vertices are not connected, unless the desired result is known by some other argument for outer automorphism groups of free products.  For homogeneous graphs, the links are always connected so this is not an issue, but  several of the questions answered in this paper remain open for non-homogeneous graphs.   Specifically, we can ask 

\begin{enumerate}
\item Is the maximal virtual derived length of a solvable subgroup of $Out(\AG)$ bounded by the dimension of $\AG$?
\item Does $Out(\AG)$ satisfy the Tits alternative?
\end{enumerate}


\def\cprime{$\prime$}

\end{document}